\documentclass[11pt]{amsart}

\usepackage{amsmath, amscd, amssymb}
\usepackage[frame,cmtip,arrow,matrix,line,graph,curve]{xy}
\usepackage{graphpap, color, paralist, pstricks}
\usepackage[mathscr]{eucal}
\usepackage[pdftex]{graphicx}
\usepackage[pdftex,colorlinks,backref=page,citecolor=blue]{hyperref}
\usepackage{tikz}

\newtheorem{theorem}{Theorem}[section]
\newtheorem{proposition}[theorem]{Proposition}
\newtheorem{corollary}[theorem]{Corollary}
\newtheorem{lemma}[theorem]{Lemma}

\theoremstyle{definition}
\newtheorem{definition}[theorem]{Definition}
\newtheorem{example}[theorem]{Example}

\newtheorem{remark}[theorem]{Remark}

\newtheorem*{acknowledgement}{Acknowledgement}

\newcommand{\PP}{\mathbb{P}}
\newcommand{\QQ}{\mathbb{Q}}
\newcommand{\CC}{\mathbb{C}}

\newcommand{\cO}{\mathcal{O} }
\newcommand{\cA}{\mathcal{A} }

\newcommand{\cI}{\mathcal{I} }
\newcommand{\cM}{\mathcal{M} }
\newcommand{\cS}{\mathcal{S} }

\newcommand{\proj}{\mathrm{Proj}\;}

\def\Mzn{\overline{\mathrm{M}}_{0,n} }

\def\Mzs{\overline{\mathrm{M}}_{0,6} }

\def\cMg{\overline{\cM}_{g} }
\def\Mt{\overline{\mathrm{M}}_{2} }
\def\cMt{\overline{\cM}_{2} }

\def\Mznpdd{\overline{\mathrm{M}}_{0,n}(\PP^{d}, d) }
\def\Vat{V_{A}^{2}}
\def\git{/\!/ }

\begin{document}

\title{Mori's program for $\overline{M}_{0,6}$ with symmetric divisors}
\date{\today}
\author{Han-Bom Moon}
\address{Department of Mathematics, Fordham University, Bronx, NY 10458}
\email{hmoon8@fordham.edu}

\begin{abstract}
We complete Mori's program with symmetric divisors for the moduli space of stable six pointed rational curves. As an application, we give an alternative proof of the complete Mori's program of the moduli space of genus two stable curves, done by Hassett.
\end{abstract}

\maketitle

%%%%%%%%%%%%%%%%%%%%%%%%%%%%%%%%%%%

\section{introduction}

Since Hassett and Hyeon initiated a study of birational geometry of moduli spaces in the viewpoint of Mori's program in \cite{Has05, HH09, HH08}, there have been a tremendous amount of results in this direction. \textbf{Mori's program} for a moduli space $M$ consists of: 1) Compute the cone of effective divisors of $M$. 2) For an effective $\QQ$-divisor $D$, find the birational model 
\[
	M(D) := \proj \bigoplus_{m \ge 0}H^{0}(M, \cO(mD)).
\]
3) Finally, study the moduli theoretic meaning of $M(D)$ and its relation with $M$. 

In this paper, we complete Mori's program with symmetric divisors for the moduli space $\Mzs$ of stable six pointed rational curves. On the effective cone of $\Mzn$, only the subcone generated by $K_{\Mzn}$ and $\psi_{i}$-classes has been studied intensively in \cite{Sim08, AS08, FS11, KM11, Moo11a}. One obstacle of the completion of Mori's full program for $\Mzn$ is that the cone of effective divisors is huge and unknown. As an initial step, we will focus on symmetric divisors, i.e., divisors which are invariant under the natural $S_{n}$-action on $\Mzn$. 

To state our result neatly, we use the interval notation from \cite{Che08b}. For two divisors $D_{1}$ and $D_{2}$, $(D_{1}, D_{2})$ is the set of divisors $aD_{1}+bD_{2}$ such that $a, b > 0$. $[D_{1}, D_{2}]$ is the set of divisors $aD_{1}+bD_{2}$ with $a, b \ge 0$. We can define $[D_{1}, D_{2})$ and $(D_{1},D_{2}]$ in a similar way. For the description of relevant divisor classes, see Definition \ref{def:divisorclasses}.

\begin{theorem}[Theorem \ref{thm:MoriprogramMzs}]\label{thm:intromainthm}
Let $D$ be a symmetric effective divisor on $\Mzs$. Then:
\begin{enumerate}
	\item If $D \in (-K_{\Mzs}, K_{\Mzs}+\frac{1}{3}\psi)$, 
	$\Mzs(D) \cong \Mzs$. 
	\item If $D \in [K_{\Mzs}+\frac{1}{3}\psi, B_{3})$, 
	$\Mzs(D) \cong 
	(\PP^{1})^{6}\git_{L} SL_{2}$ with the symmetric linearization $L$. 
	\item If $D \in (B_{2}, -K_{\Mzs}]$, $\Mzs(D)$ is 
	Veronese quotient 
	$\Vat$ with symmetric weight data $A = (\frac{1}{2}, \cdots, 
	\frac{1}{2})$. 
	\item Both $\Mzs(B_{2})$ and $\Mzs(B_{3})$ are a point. 
\end{enumerate}
\end{theorem}

It is interesting that both $(\PP^{1})^{6}\git_{L} SL_{2}$ and $\Vat$ are classically known varieties. $(\PP^{1})^{6}\git_{L}SL_{2}$ is isomorphic to \textbf{Segre cubic} $\cS_{3}$ (Remark \ref{rem:Segrecubic}) and $\Vat$ is isomorphic to \textbf{Igusa quartic} $\cI_{4}$ (Remark \ref{rem:Igusaquartic}) or \textbf{Castelnuovo-Richmond quartic}. Also $\cI_{4}$ is isomorphic to Satake compactification $\overline{\cA}_{2}(2)$ of the moduli space of principally polarized abelian surfaces with level two structures (\cite{Igu64}). Moreover, these two varieties are known to be projectively dual to each other (\cite[Remark I.3]{DO88}). 

All birational models appear here are classically known varieties with or without their modular interpretations. For example, see \cite[Section 9.4]{Dol12} and references therein. So this article is an interpretation of the relation between classically known varieties using a modern viewpoint of Mori's program. Also Item (1) and (2), which are on the direction towards canonical divisor, are proved in \cite{Sim08, FS11, AS08, KM11, Moo13}.

The author wants to point out a simple but important observation. As we can see in the definition of $\Vat$ in Section \ref{sec:veronesequotient}, the birational model $\mathcal{I}_{4}$ is not a moduli space of abstract pointed curves, but that of (equivalent classes of) configurations of points. In general to find birational models of $\Mzn$ in the direction towards the anti-canonical divisor, it is insufficient to study moduli spaces of rational curves with worse singularities. The reason is that all moduli spaces pointed rational curves (with allowing worse singularities) are contractions of $\Mzn$ (\cite{Smy13}), as opposed to the case of the moduli stack $\cMg$. In the viewpoint of symmetric Mori's program, $\Mzs$ we discuss here is very simple in the sense that there is no flip. Thus we can explain everything by using well-known contractions. But for $n \ge 7$, there must be several flips even for symmetric divisors. So to understand Mori's program for $\Mzn$ for larger $n$, we need to `find' completely new modular interpretations of birational models. 

As a quick application of our results, we give a complete description of Mori's program for $\cMt$. All smooth genus two curves are hyperelliptic, thus the coarse moduli space $\Mt$ of the moduli space of genus two stable curves is isomorphic to $\Mzs/S_{6}$ (\cite[Corollary 2.5]{AL02}). Therefore we can directly translate symmetric Mori's program, as Mori's program for $\Mt$ and that for $\cMt$. The investigation of Mori's program for $\cMt$ was done by Hassett in \cite{Has05}, as an initial step of Hassett-Keel program. It has been one of the most influential projects on the birational geometry of moduli spaces in the last several years. As a consequence of Theorem \ref{thm:intromainthm}, we give a different proof of Hassett's theorem (\cite[Theorem 4.10]{Has05}) for $\cMt$.

\begin{theorem}[Theorem \ref{thm:MoriprogramcMt}]\label{thm:introtheoremcMt}
Let $D$ be an effective divisor on $\cMt$. Then:
\begin{enumerate}
	\item If $D \in (\lambda, \delta_{0}+12\delta_{1})$, 
	$\cMt(D) \cong \Mt$.
	\item If $D \in [\delta_{0}+12\delta_{1}, \delta_{1})$, 
	$\cMt(D) \cong \PP^{6}\git SL_{2}$.
	\item If $D \in (\delta_{0}, \lambda]$, $\cMt(D)$ is the
	Satake compactification $\overline{\cA}_{2}^{\mathrm{Sat}}$ of 
	the moduli space $\cA_{2}$ of principally polarized abelian surfaces.
	\item Both $\cMt(\delta_{0})$ and $\cMt(\delta_{1})$ are a point.
\end{enumerate}
\end{theorem}

We can summarize Mori's program for $\Mzs$ and $\cMt$ with Figure \ref{fig:Moriprogram}. The diagonal maps are divisorial contractions (contracted divisors are indicated on arrows) and vertical maps are $S_{6}$-quotient. 

\begin{figure}[!ht]
\[
	\xymatrix{&\Mzs\ar_{B_{2}}[ld]\ar[dd]\ar^{B_{3}}[rd]\\
	\cI_{4} = \overline{\cA}_{2}(2) = \Vat\ar[dd]&&
	(\PP^{1})^{6}\git_{L}SL_{2} = \cS_{3} \ar[dd]\\
	& \Mt \ar_{\Delta_{0}}[ld]\ar^{\Delta_{1}}[rd]\\
	\overline{\cA}_{2}^{\mathrm{Sat}} && \PP^{6}\git SL_{2}}
\]
\caption{Mori's program for $\Mzs$ and $\cMt$}\label{fig:Moriprogram}
\end{figure}

After the author finish the preparation of this manuscript, he noticed that Lange and Ortega recently ran the log minimal model program for the moduli space of even spin curves of genus 2 in \cite{LO13}. One can regard the main result of this paper as a theorem parallel to \cite[Theorem 2]{LO13} with additional monodromies. 

This paper is organized as follows. In Section \ref{sec:divcurve}, we review basic facts about divisor and curve classes on $\Mzn$. Section \ref{sec:veronesequotient} is for the background about Veronese quotients. We will study the geometry of a particular Veronese quotient $\Vat$ with symmetric weight data $A$ in Section \ref{sec:explicitcomputation}. By using them, we prove Theorem \ref{thm:intromainthm} on Section \ref{sec:proofmaththm}. Section \ref{sec:MoriprogramcMt} is a proof of Theorem \ref{thm:introtheoremcMt}.

We will work over an algebraically closed field of characteristic 0. 

\begin{acknowledgement}
The author would like to thank Angela Gibney for thorough and helpful comments on an earlier draft of this paper. Also he thanks Igor Dolgachev for indicating some references. 
\end{acknowledgement}

\section{Divisors and curves on $\Mzn$}\label{sec:divcurve}

We begin by reviewing general facts about divisors and curves on $\Mzn$. The moduli space $\Mzn$ of stable $n$-pointed rational curves is a smooth projective variety of dimension $n-3$ with a natural $S_{n}$-action permuting marked points. A divisor $D$ on $\Mzn$ is called symmetric if it is invariant under the $S_{n}$ action. The Neron-Severi vector space $N^{1}(\Mzn)$ has dimension $2^{n-1} - {n \choose 2} - 1$, but its $S_{n}$-invariant part $N^{1}(\Mzn)^{S_{n}} \cong N^{1}(\Mzn/S_{n})$ is $\lfloor n/2\rfloor - 1$ dimensional (\cite[Theorem 1.3]{KM96}). The following is a list of natural symmetric divisors on $\Mzn$. 

\begin{definition}\label{def:divisorclasses}
\begin{enumerate}
	\item For $i = 2, 3, \cdots, n-2$, 
	let $B_{i}$ be the closure of the locus of 
	curves $C$ with two irreducible components $C_{1}$ and $C_{2}$
	such that 
	$C_{1}$ (resp. $C_{2}$) contains $i$ (resp. $n-i$) marked points.
	$B_{i}$ is called a symmetric boundary divisor. 
	By definition, $B_{i} = B_{n-i}$.
	Let $B = \sum_{i=2}^{\lfloor n/2 \rfloor}B_{i}$ be the total boundary 
	divisor. 
	\item Fix $1 \le i \le n$. 
	Let $\mathbb{L}_{i}$ be a line bundle on $\Mzn$ such that 
	over $(C, x_{1}, \cdots, x_{n}) \in \Mzn$, the fiber is 
	$\Omega_{C, x_{i}}$, the cotangent space of $C$ at $x_{i}$. 
	Let $\psi_{i} = c_{1}(\mathbb{L}_{i})$. 
	If we denote $\psi = \sum_{i=1}^{n}\psi_{i}$, then $\psi$ is an 
	$S_{n}$-invariant divisor. 
	\item Let $K_{\Mzn}$ be the canonical divisor of $\Mzn$. 
\end{enumerate}
\end{definition}

The symmetric effective cone $\mathrm{Eff}(\Mzn)^{S_{n}} \cong \mathrm{Eff}(\Mzn/S_{n})$ is generated by symmetric boundary divisors (\cite[Theorem 1.3]{KM96}). Thus we can write down $K_{\Mzn}$ and $\psi$ as nonnegative linear combinations of boundary divisors. 

\begin{lemma}\cite[Proposition 2]{Pan97}, \cite[Lemma 2.9]{Moo11a}
On $N^{1}(\Mzn)$, the following relations hold.
\begin{enumerate}
	\item 
	$\displaystyle K_{\Mzn} = \sum_{i=2}^{\lfloor n/2\rfloor}
	\left(\frac{i(n-i)}{n-1} - 2\right)B_{i}$.
	\item $\displaystyle \psi = K_{\Mzn} + 2B$.
\end{enumerate}
\end{lemma}

Now we move to curve classes on $\Mzn$. Let $S_{1} \sqcup S_{2} \sqcup S_{3} \sqcup S_{4} = \{1,2, \cdots, n\}$ be a partition. Let $F(S_{1}, S_{2}, S_{3}, S_{4})$ be an F-curve class corresponding to the partition (\cite[Section 4]{KM96}).

\begin{lemma}\cite[Corollary 4.4]{KM96}
Let $F = F(S_{1}, S_{2}, S_{3}, S_{4})$ be an F-curve and let $a_{j} = |S_{j}|$. Then 
\[
	F \cdot \sum r_{i} B_{i} = - r_{a_{1}}-r_{a_{2}}-r_{a_{3}}-r_{a_{4}}
	+ r_{a_{1}+a_{2}} + r_{a_{1}+a_{3}}+r_{a_{1}+a_{4}}
\]
if we define $r_{1} = 0$ and $r_{a+b} = r_{n-a-b}$.
\end{lemma}

We need to know another curve class $C_{j}$ (see \cite[Lemma 4.8]{KM96}). Fix a $j$-pointed $\PP^{1}$. And let $x$ be a moving point on $\PP^{1}$. By gluing fixed $n-j+1$ pointed $\PP^{1}$ whose last marked point is $y$ to the $j$-pointed $\PP^{1}$ along $x$ and $y$ and stabilizing it, we obtain an one parameter family of $n$-pointed stable curves over $\PP^{1}$, i.e., a curve $C_{j} \cong \PP^{1}$ on $\Mzn$.

\begin{lemma}\cite[Lemma 4.8]{KM96}
\[
	C_{j}\cdot B_{i} = \begin{cases}
	j, & i = j-1,\\ -(j-2), & i = j,\\ 0, & \mbox{otherwise.}
	\end{cases}
\]
\end{lemma}

For convenience of readers, we leave a special case of $\Mzs$ below. The proof is an easy combination of above results. 

\begin{corollary}
The $S_{6}$-invariant Neron-Severi space $N^{1}(\Mzs)^{S_{6}}$ has dimension two. The symmetric effective cone $\mathrm{Eff}(\Mzs)^{S_{6}}$ is generated by $B_{2}$ and $B_{3}$. Moreover,
\begin{enumerate}
	\item $K_{\Mzs} = -\frac{2}{5}B_{2} - \frac{1}{5}B_{3}$,
	\item $\psi = \frac{8}{5}B_{2} + \frac{9}{5}B_{3}$,
	\item $B_{2} = -\frac{9}{2}K_{\Mzs} - \frac{1}{2}\psi$,
	\item $B_{3} = 4K_{\Mzs} + \psi$.
\end{enumerate}
\end{corollary}

Figure \ref{fig:effectivecone} shows several rays in $N^{1}(\Mzs)^{S_{6}}$ generated by special divisors.

\begin{figure}[!hb]
\begin{tikzpicture}[scale=0.5]
	\draw[->][line width=1.5pt] (5, 0) -- (10, 0);
	\draw[line width=1.5pt] (0,0) -- (5,0);
	\draw[->][line width=1.5pt] (5,0) -- (5,5);
	\draw[->][line width=1.5pt] (5,0) -- (9.8,2);
	\draw[->][line width=1.5pt] (5,0) -- (0.2, -1);
	\node at (11,0) {$K_{\Mzs}$};
	\node at (5,5.5) {$\psi$};
	\node at (10.5, 2) {$B_{3}$};
	\node at (-0.5,-1) {$B_{2}$};
\end{tikzpicture}
\caption{The Neron-Severi space of $\Mzs$}\label{fig:effectivecone}
\end{figure}

On $\Mzs$, there are only two types of F-curves, whose partition is of the form $1+1+1+3$ or $1+1+2+2$. We will denote the corresponding F-curves by $F_{1,1,1,3}$ and $F_{1,1,2,2}$ respectively. 

\begin{corollary}\label{cor:intersection}
On $\Mzs$, the intersection of symmetric divisors and curve classes are given by
the following table.
\begin{center}
\begin{tabular}{|c|c|c|c|c|}
\hline
& $\psi$ & $K_{\Mzs}$ & $B_{2}$ & $B_{3}$\\ \hline
$F_{1,1,1,3}$ & 3 & -1 & 3 & -1\\ \hline
$F_{1,1,2,2}$ & 2 & 0 & -1 & 2\\ \hline
$C_{4}$ & 4 & 0 & -2 & 4\\ \hline
\end{tabular}
\end{center}
\end{corollary}

Note that on $\Mzs$, $C_{3} = F_{1,1,1,3}$.

\section{Veronese quotients and their geometric properties}\label{sec:veronesequotient}

In this section, we give a review about (a special case) of Veronese quotients introduced in \cite{Gia10}. See \cite{GJM13, GJMS12} for a generalization. The following description is different from the original one in \cite{Gia10} (However, see \cite[Remark 2.4]{Gia10}). We will use a construction using moduli spaces of stable maps, which is useful for our purpose, in particular the description of the morphism $\varphi_{A} : \Mzn \to V_{A}^{d}$. 

\subsection{Veronese quotients}
Let $\Mznpdd$ be Kontsevich's moduli space of stable maps (\cite{FP97}). It parametrizes maps $f : (C, x_{1}, \cdots, x_{n}) \to \PP^{d}$ from an arithmetic genus 0 curve $C$ with $n$ marked points to $\PP^{d}$ such that $f_{*}[C] = d$ with following stability conditions. Such map $f$ is called stable if
\begin{itemize}
	\item $C$ has at worst nodal singularities,
	\item $x_{i}$ are distinct smooth points on $C$, 
	\item $\omega_{C} + \sum x_{i}$ is $f$-ample. 
\end{itemize}

There are $n$ evaluation maps $e_{i} : \Mznpdd \to \PP^{d}$. By taking the product of these maps, we have a map
\[
	e : \Mznpdd \to (\PP^{d})^{n}.
\]
Let $U_{d,n}$ be the image of $e$. Note that $SL_{d+1}$ acts on both $\Mznpdd$ and $(\PP^{d})^{n}$ via $SL_{d+1} \to \mathrm{Aut}(\PP^{d})$ and $e$ is $SL_{d+1}$-equivariant. Thus $U_{d,n}$ is an $SL_{d+1}$-invariant subvariety of $(\PP^{d})^{n}$.

For a choice of positive rational numbers $A = (a_{1}, \cdots, a_{n})$, we can construct a $\QQ$-linearization $L_{A} := \cO(a_{1})\otimes \cO(a_{2}) \otimes \cdots \otimes \cO(a_{n})$. Since the stability does not change if we replace $A$ by its scalar multiple, we will normalize it as $\sum_{i}a_{i} = d+1$. For $(\PP^{d})^{n}$, the (semi)stability can be computed explicitly. 

\begin{theorem}\cite[Theorem 11.1]{Dol03}
Let $A = (a_{1}, \cdots, a_{n})$ be a normalized linearization. A configuration $(x_{1}, \cdots, x_{n}) \in (\PP^{d})^{n}$ is (semi)stable if and only if for any proper linear subspace $W \subset \PP^{d}$, 
\[
	\sum_{x_{j}\in W}a_{j}\; (\le) < \dim W + 1.
\]
\end{theorem}\label{thm:stability}

In particular, to guarantee the nonemptiness of (semi)stable locus, we need a necessary condition $a_{i} \;(\le) < 1$ for all $i$. Thus the hypersimplex
\[
	\Delta(d+1, n) = \{(a_{1}, \cdots, a_{n}) \in \QQ^{n}\;|\;
	0 \le a_{i} \le 1, \sum_{i=1}^{n}a_{i} = d+1\}
\]
can be regarded as the space of effective linearizations.

\begin{definition}
Let $A = (a_{1}, \cdots, a_{n}) \in \Delta(d+1, n)$ such that $n \ge d+3$.
The \textbf{Veronese quotient} is the GIT quotient
\[
	V_{A}^{d} := U_{d,n}\git_{L_{A}}SL_{d+1}.
\]
\end{definition}

\begin{remark}
\begin{enumerate}
	\item It is called Veronese quotient because for a general configuration 
	$(x_{1}, \cdots, x_{n})$ with $n \ge d+3$, 
	there exists a unique rational normal curve $C$ in $\PP^{d}$ such that 
	$x_{i} \in C$ for all $i$. 
	\item This is a special case $\gamma = 0$ of general Veronese quotients 
	described in \cite{GJM13, GJMS12}.
	\item Up to projective equivalence, there is a unique rational normal 
	curve in $\PP^{d}$. Thus after taking the quotient, we can regard it as a 
	moduli space of configuration of points on an abstract rational curve and 
	their degenerations. So $V_{A}^{d}$ is birational to $\Mzn$.
\end{enumerate}
\end{remark}

\begin{example}\label{ex:Kapranovmorphism}
If $d = 1$, then $U_{1,n} = (\PP^{1})^{n}$ and the GIT quotient $(\PP^{1})^{n}\git_{L_{A}} SL_{2}$ itself is birational to $\Mzn$. In this case, the existence of a birational morphism $\rho: \Mzn \to (\PP^{1})^{n}\git_{L_{A}} SL_{2}$ is proved in \cite{Kap93b}. 

When $n$ is even and $L_{A}$ is a symmetric linearization, $(\PP^{1})^{n}\git SL_{2}$ has ${n \choose 2}$ singular points. When $n=6$, the map $\rho$ contracts an irreducible component of a boundary divisor $B_{3}$ to a singular point. In particular, $F_{1,1,1,3}$ is contracted. Thus the semi-ample divisor $\rho^{*}L$ for an ample divisor $L$ on $(\PP^{1})^{6}\git_{L_{A}}SL_{2}$ is a scalar multiple of $K_{\Mzs}+\frac{1}{3}\psi$ (see Corollary \ref{cor:intersection}).
\end{example}

\begin{example}
For the purpose of this paper, the most important example is $\Vat$, where $n = 6$ and $A = (\frac{1}{2}, \cdots, \frac{1}{2})$. In this case there are many strictly semi-stable points on $\Vat$. We will study its (semi)stability in Section \ref{sec:explicitcomputation} in detail.
\end{example}

\subsection{Morphisms from $\Mzn$ and canonical polarizations}

One interesting common property of Veronese quotients is that they admit  birational morphisms from $\Mzn$. This section is an outline of a proof in \cite{GJM11}. You can find an original proof via Chow quotients in \cite[Section 3]{Gia10}.

For $(f : (C, x_{1}, \cdots, x_{n}) \to \PP^{d}) \in \Mznpdd$, by forgetting the map $f$ and stabilizing the domain, we can obtain a stable rational curve $(C^{s}, x_{1}, \cdots, x_{n}) \in \Mzn$. Thus there is a forgetful morphism $F : \Mznpdd \to \Mzn$.
\[
	\xymatrix{\Mznpdd \ar[r]^{\;\;\;e} \ar[d]_{F} & (\PP^{d})^{n}\\
	\Mzn}
\]

For an effective linearization $L_{A}$ on $(\PP^{d})^{n}$, there exists an effective linearization $L_{A}'$ on $\Mznpdd$ such that 
\[
	e^{-1}(((\PP^{d})^{n})^{s}(L_{A}))  
	\subset \Mznpdd^{s}(L_{A}) \subset 
	\Mznpdd^{ss}(L_{A}') \subset 
	e^{-1}(((\PP^{d})^{n})^{ss}(L_{A}'))
\]
where $X^{ss}(L)$ (resp. $X^{s}(L)$) is the semistable (resp. stable) part of $X$ with respect to the linearization $L$. In particular, we have a quotient morphism $\overline{e} : \Mznpdd\git_{L_{A}'}SL_{d+1} \to (\PP^{d})^{n}\git_{L_{A}}SL_{d+1}$. 

Since $F$ is $SL_{d+1}$-invariant, there exists a quotient map $\overline{F}$.
\[
	\xymatrix{\Mznpdd\git_{L_{A}'}SL_{d+1} \ar[r]^{\;\;\;\overline{e}}
	\ar[d]_{\overline{F}}& (\PP^{d})^{n}\git_{L_{A}}SL_{d+1}\\
	\Mzn}
\]
In \cite[Proposition 4.6]{GJM13}, it is proved that for a general effective linearization (stability and semistability  on $(\PP^{d})^{n}$ coincide for $L_{A}$), $\overline{F}$ is an isomorphism. Thus we have a morphism 
\[
	\varphi_{A} = \overline{ev} \circ \overline{F}^{-1}: 
	\Mzn \to (\PP^{d})^{n}\git_{L_{A}}SL_{d+1}.
\]
It is straightforward to check that the image of $\varphi_{A}$ is $V_{A}^{d}$. 

For any effective linearization $L_{A}$, if we perturb it slightly, we obtain an effective linearization $L_{A_{\epsilon}}$ such that the stability coincides with the semi-stability. From the general theory of the variation of GIT, we have a morphism 
\[
	\Mznpdd\git_{L_{A_{\epsilon}}'}SL_{d+1} 
	\to (\PP^{d})^{n}\git_{L_{A_{\epsilon}}}SL_{d+1} 
	\to (\PP^{d})^{n}\git_{L_{A}}SL_{d+1}.
\]
We will denote it by $\overline{e}$, too.
Also $\varphi_{A}$ is defined as $\overline{e} \circ \overline{F}^{-1}$. 

\begin{remark}\label{rem:contractionmap}
This morphism $\varphi_{A}$ can be described in the following slightly different way. Note that for any effective linearization $L_{A}$ on $(\PP^{d})^{n}$, there is a commutative diagram 
\[
	\xymatrix{\Mznpdd^{s}(L_{A_{\epsilon}}') \ar@{^{(}->}[r]\ar[d]&
	\Mznpdd^{ss}(L_{A}') \ar[d]\\
	((\PP^{d})^{n})^{s}(L_{A_{\epsilon}})\ar@{^{(}->}[r]&
	((\PP^{d})^{n})^{ss}(L_{A})}
\]
where $L_{A}'$, $L_{A_{\epsilon}}'$ are linearizations explained above. Finally,  we have a quotient diagram
\[
	\xymatrix{\Mznpdd\git_{L_{A_{\epsilon}}'}SL_{d+1}
	\ar[r]\ar[d]&
	\Mznpdd \git_{L_{A}'}SL_{d+1} \ar[d]\\
	((\PP^{d})^{n})\git_{L_{A_{\epsilon}}}SL_{d+1}\ar[r]&
	((\PP^{d})^{n})\git_{L_{A}}SL_{d+1}}
\]
which is commutative.
	
This result gives us a practical way to describe the contraction map $\varphi_{\cA}$. Fix a curve $(C, x_{1}, \cdots, x_{n}) \in \Mzn$. First of all, find any stable map $f : (\tilde{C}, x_{1}, \cdots, x_{n}) \to \PP^{d}$ up to projective transformation such that $f \in \Mznpdd^{ss}$ with respect to $L_{A'}$ and $F(f) = (C, x_{1}, \cdots, x_{n})$. (If there is a strictly semistable point, then $f$ and $\tilde{C}$ may not be unique.) Indeed, $f$ can be determined by degree of $f$ on each irreducible component of $\tilde{C}$. Then we obtain a point configuration $x = (f(x_{1}), \cdots, f(x_{n})) \in (\PP^{d})^{n}$. By the variation of GIT, there exists a unique closed orbit in $((\PP^{d})^{n})^{ss}$ with respect to $L_{A}$ which is contained in the closure of the orbit of $x$. This is $\varphi_{A}(C, x_{1}, \cdots, x_{n})$. For a detail, see \cite[Section 3]{GJM13}.
\end{remark}

A GIT quotient $V_{A}^{d}$ has a canonical polarization $\overline{L}_{A}$ from its definition. Since $\varphi_{A} : \Mzn \to V_{A}^{d}$ is a regular morphism, by taking pull-back we obtain a semi-ample line bundle $D_{A}:= \varphi_{A}^{*}(\overline{L}_{A})$ on $\Mzn$.

The numerical class of $D_{A}$ is computed in \cite[Theorem 2.1]{GJMS12} in a broader context (In our situation, $\gamma = 0$ in the statement of the Theorem.). The following result is a special case we want to see in this article. 

\begin{lemma}\label{lem:canpolarization}
Suppose that $n = 6$ and $A = (\frac{1}{2}, \cdots, \frac{1}{2})$. Consider $\Vat$ and the pull-back $D_{A}=\varphi_{A}^{*}\overline{L}_{A}$ of the canonical polarization. Then $F_{1,1,1,3}\cdot D_{A} = \frac{1}{2}$ and $F_{1,1,2,2} \cdot D_{A} = 0$. 
\end{lemma}

\section{An explicit computation of Veronese quotient $\Vat$}
\label{sec:explicitcomputation}

When $n=6$, with respect to the symmetric linearization $L_{A}$, $U_{2, 6}$ has strictly semistable points. Thus to describe $V_{A}^{2} = U_{2, 6}\git_{L_{A}}SL_{3}$ concretely, we need to analyze the stability of $(\PP^{2})^{6}$ in detail. In this section, by computing the (semi)stable locus, we describe the morphism $\varphi_{A} : \Mzs \to V_{A}^{2}$ explicitly. From this section, we will use symmetric linearizations only. Note that there exists a unique symmetric linearization on $(\PP^{r})^{d}$ up to normalization. So we will not indicate the linearization for GIT quotients.

\subsection{An explicit computation of stability on $(\PP^{2})^{6}\git SL_{3}$}
Due to Theorem \ref{thm:stability}, for a strictly semistable configuration on $(\PP^{2})^{6}$ there are two possibilities:
\begin{itemize}
	\item on a point, there are exactly two marked points;
	\item on a line, there are four points.
\end{itemize}

Thus we can make a list of strictly semistable configurations. See Table \ref{tbl:semistableconfiguration}. For each stratum, there is a figure for a typical element in the stratum. Three lines on the figure are standard coordinates lines on $\PP^{2}$ and the symbol $\odot$ means a point with multiplicity two. In the next three rows, the stabilizer in $SL_{3}$ of the configuration, the dimension and the orbit closure in the semistable locus for each stratum are written. 

\begin{table}[!ht]
\begin{tabular}{|c|c|c|c|c|}\hline
stratum & I & II & III & IV\\ \hline
&
\begin{tikzpicture}[scale=0.2]
	\draw[line width = 1pt] (1, 10) -- (1, 0);
	\draw[line width = 1pt] (0, 1) -- (10, 1);
	\draw[line width = 1pt] (0, 10) -- (10, 0);
	\fill (1,1) circle (10pt);
	\fill (1,9) circle (10pt);
	\fill (9,1) circle (10pt);
	\draw (1,1) circle (18pt);
	\draw (1,9) circle (18pt);
	\draw (9,1) circle (18pt);
\end{tikzpicture} &
\begin{tikzpicture}[scale=0.2]
	\draw[line width = 1pt] (1, 10) -- (1, 0);
	\draw[line width = 1pt] (0, 1) -- (10, 1);
	\draw[line width = 1pt] (0, 10) -- (10, 0);
	\fill (1,1) circle (10pt);
	\fill (1,9) circle (10pt);
	\fill (9,1) circle (10pt);
	\fill (1,5) circle (10pt);
	\draw (1,1) circle (18pt);
	\draw (9,1) circle (18pt);
\end{tikzpicture} &
\begin{tikzpicture}[scale=0.2]
	\draw[line width = 1pt] (1, 10) -- (1, 0);
	\draw[line width = 1pt] (0, 1) -- (10, 1);
	\draw[line width = 1pt] (0, 10) -- (10, 0);
	\fill (1,1) circle (10pt);
	\fill (1,9) circle (10pt);
	\fill (9,1) circle (10pt);
	\fill (3,4) circle (10pt);
	\draw (1,1) circle (18pt);
	\draw (9,1) circle (18pt);
\end{tikzpicture} &
\begin{tikzpicture}[scale=0.2]
	\draw[line width = 1pt] (1, 10) -- (1, 0);
	\draw[line width = 1pt] (0, 1) -- (10, 1);
	\draw[line width = 1pt] (0, 10) -- (10, 0);
	\fill (1,1) circle (10pt);
	\fill (1,9) circle (10pt);
	\fill (9,1) circle (10pt);
	\fill (1,5) circle (10pt);
	\fill (5,1) circle (10pt);
	\draw (1,1) circle (18pt);
\end{tikzpicture}
\\ \hline
stabilizer & $(\CC^{*})^{2}$ & $\CC^{*}$ & 1 & 1\\ \hline
dimension & 6 & 7 & 8 & 8\\ \hline
orbit closure & closed & $\in$ I & $\in$ I, II & $\in$ I, II\\ \hline
\hline
stratum & V & VI & VII & VIII\\ \hline
&\begin{tikzpicture}[scale=0.2]
	\draw[line width = 1pt] (1, 10) -- (1, 0);
	\draw[line width = 1pt] (0, 1) -- (10, 1);
	\draw[line width = 1pt] (0, 10) -- (10, 0);
	\fill (1,1) circle (10pt);
	\fill (1,9) circle (10pt);
	\fill (9,1) circle (10pt);
	\fill (1,5) circle (10pt);
	\fill (5,5) circle (10pt);
	\draw (1,1) circle (18pt);
\end{tikzpicture}
&
\begin{tikzpicture}[scale=0.2]
	\draw[line width = 1pt] (1, 10) -- (1, 0);
	\draw[line width = 1pt] (0, 1) -- (10, 1);
	\draw[line width = 1pt] (0, 10) -- (10, 0);
	\fill (1,1) circle (10pt);
	\fill (1,9) circle (10pt);
	\fill (9,1) circle (10pt);
	\fill (1,5) circle (10pt);
	\fill (3,3) circle (10pt);
	\draw (1,1) circle (18pt);
\end{tikzpicture}
&
\begin{tikzpicture}[scale=0.2]
	\draw[line width = 1pt] (1, 10) -- (1, 0);
	\draw[line width = 1pt] (0, 1) -- (10, 1);
	\draw[line width = 1pt] (0, 10) -- (10, 0);
	\fill (1,1) circle (10pt);
	\fill (1,9) circle (10pt);
	\fill (9,1) circle (10pt);
	\fill (6,4) circle (10pt);
	\fill (4,6) circle (10pt);
	\draw (1,1) circle (18pt);
\end{tikzpicture}
&
\begin{tikzpicture}[scale=0.2]
	\draw[line width = 1pt] (1, 10) -- (1, 0);
	\draw[line width = 1pt] (0, 1) -- (10, 1);
	\draw[line width = 1pt] (0, 10) -- (10, 0);
	\fill (1,1) circle (10pt);
	\fill (1,9) circle (10pt);
	\fill (9,1) circle (10pt);
	\fill (4,6) circle (10pt);
	\fill (4.5,4) circle (10pt);
	\draw (1,1) circle (18pt);
\end{tikzpicture}
\\ \hline 
stabilizer & 1 & 1 & $\CC^{*}$ & 1\\ \hline
dimension & 9 & 9 & 8 & 9\\ \hline 
orbit closure & $\in$ I, II & $\in$ I, II & closed & $\in$ VII \\ \hline \hline 
stratum & IX & X & XI & \\ \hline
& 
\begin{tikzpicture}[scale=0.2]
	\draw[line width = 1pt] (1, 10) -- (1, 0);
	\draw[line width = 1pt] (0, 1) -- (10, 1);
	\draw[line width = 1pt] (0, 10) -- (10, 0);
	\fill (1,1) circle (10pt);
	\fill (1,9) circle (10pt);
	\fill (9,1) circle (10pt);
	\fill (3,5) circle (10pt);
	\fill (5,3) circle (10pt);
	\draw (1,1) circle (18pt);
\end{tikzpicture}
&
\begin{tikzpicture}[scale=0.2]
	\draw[line width = 1pt] (1, 10) -- (1, 0);
	\draw[line width = 1pt] (0, 1) -- (10, 1);
	\draw[line width = 1pt] (0, 10) -- (10, 0);
	\fill (1,1) circle (10pt);
	\fill (1,9) circle (10pt);
	\fill (9,1) circle (10pt);
	\fill (6,4) circle (10pt);
	\fill (4,6) circle (10pt);
	\fill (1,5) circle (10pt);
\end{tikzpicture}
&
\begin{tikzpicture}[scale=0.2]
	\draw[line width = 1pt] (1, 10) -- (1, 0);
	\draw[line width = 1pt] (0, 1) -- (10, 1);
	\draw[line width = 1pt] (0, 10) -- (10, 0);
	\fill (1,1) circle (10pt);
	\fill (1,9) circle (10pt);
	\fill (9,1) circle (10pt);
	\fill (6,4) circle (10pt);
	\fill (4,6) circle (10pt);
	\fill (3,3) circle (10pt);
\end{tikzpicture}
&
\\ \hline
stabilizer & 1 & 1 & 1 & \\ \hline
dimension & 10 & 9 & 10 & \\ \hline
orbit closure & $\in$ VII & $\in$ VII & $\in$ VII & \\ \hline
\end{tabular}
\vskip 0.3cm
\caption{Strictly semistable configurations}
\label{tbl:semistableconfiguration}
\end{table}

Since the set of geometric points on a GIT quotient bijectively corresponds to the set of closed orbits in the semistable locus, the set of geometric points on $\Vat$ is in bijection with the orbits in $\mathrm{I} \sqcup  \mathrm{VII} \sqcup U_{2,6}^{s}$.

\subsection{A description of $\varphi_{A}$}
Now we can explicitly describe the morphism $\varphi_{A} : \Mzs \to \Vat$ by following the recipe in Remark \ref{rem:contractionmap}. For any $(C, x_{1}, \cdots, x_{6}) \in \Mzs - (B_{2}\cup B_{3})$, there is a degree 2 map $f : C \to \PP^{2}$ whose image is a nonsingular conic. Note that all nonsingular conics are projectively equivalent. $\varphi_{A}(C, x_{1}, \cdots, x_{6})$ is the image (up to projective equivalence) of six points. 

For $(C, x_{1}, \cdots, x_{6}) \in B_{3} - B_{3}\cap B_{2}$, if we define a map $f : C \to \PP^{2}$ such that $\deg f = 1$ on each irreducible component of $C$ and the image is a union of two distinct lines, $f$ is stable with respect to  $L_{A'}$. So the image of each irreducible component is a line. Therefore $\varphi_{A}(C, x_{1}, \cdots, x_{6})$ is a configuration of distinct points on the union of two lines such that 1) on each line there are three distinct points, and 2) on the intersection of two lines there is no marked point. 

An interesting contraction happens on $B_{2}$. Let $(C = C_{1} \cup C_{2}, x_{1}, \cdots, x_{6})$ be a general point on $B_{2}$. Without loss of generality, suppose that $x_{1}, \cdots, x_{4}$ are on $C_{1}$. Then $\deg f|_{C_{1}}$ may be one or two. If it is one, $\deg f|_{C_{2}} = 1$ so the image is a configuration of two lines such that on one line there are four distinct points and on the other line there are two distinct points. Also there is no marked point on the intersection of two lines. Thus it is an element of the stratum XI. If $\deg f|_{C_{2}} = 2$, then the image is a conic with five points (one of them has multiplicity two). So the image is a configuration of type IX. In any cases, the image has its orbit closure on the stratum VII. Note that $(C, x_{1}, \cdots, x_{6})$ depends on the cross ratios of five points, $x_{1}, \cdots, x_{4}$ and the singular points. But on the stratum VII, it depends only on the cross ratio of four points $x_{1}, \cdots, x_{4}$. Therefore the image has dimension one and the contracted curve is exactly $C_{4}$. Since the cross ratio of four points parameterized by $\PP^{1}$, an irreducible component of $B_{2}$ maps to $\PP^{1}$ by $\varphi_{A}$. 

We can observe the contraction of $F_{1,1,2,2}$ in a similar way. Let $(C = C_{1} \cup C_{2} \cup C_{3}, x_{1}, \cdots, x_{6})$ is a general element on $F_{1,1,2,2}$ where $C_{1}$ and $C_{2}$ are two tails. Then $(\deg f|_{C_{1}}, \deg f|_{C_{2}})$ can be $(1,1), (1, 0), (0, 1)$, or $(0,0)$. In each case, it is straightforward to check that the image configuration is of type IV, VI, VI, and III respectively. Thus in any cases, the orbit closure contains the stratum I. Since it does not have a moduli, $F_{1,1,2,2}$ is contracted.

\begin{proposition}\label{prop:geometryofquotient}
The contraction map $\varphi_{A}: \Mzs \to \Vat$ is an isomorphism outside of $B_{2}$. The image of $B_{2}$ is the union of 15 projective lines $L_{1}, \cdots, L_{15}$. Each $L_{i}$ intersects other $L_{j}$ at three points, and at each intersection point there are three $L_{i}$'s which passes through it. Finally, $\Vat$ is singular along $\cup L_{i}$. 
\end{proposition}

\begin{proof}
The first statement is already discussed above. Note that there are 15 irreducible components of $B_{2}$. Each of them is contracted to a line, so the image is a union of 15 lines $L_{1}, \cdots, L_{15}$. An irreducible component intersects with other irreducible components of $B_{2}$ along three projective lines. They are F-curves $F_{1,1,2,2}$ and so are contracted. Note that $F_{1,1,2,2}$ has a point which is an intersection of three irreducible components of $B_{2}$. Therefore for each intersection point there are three $L_{i}$'s. 

Now it is enough to prove the last statement about the singularity. For a general curve $C_{4}$ in $B_{2}$, an irreducible component of $B_{2}$ containing $C_{4}$ is isomorphic to $\overline{\mathrm{M}}_{0,5}$. Let $p$ be the moving point on $\PP^{1}$ which is used to define $C_{4}$ (see Section \ref{sec:divcurve}). Let $q_{p}: \overline{\mathrm{M}}_{0,5} \to \PP^{2}$ be Kapranov morphism for the marked point $p$ (\cite[Section 4.2]{Kap93b}). Then $q_{p}(C_{4})$ is a conic on $\PP^{2}$. Therefore by \cite[Lemma 4.5]{KM96}, 
\[
	N_{D_{2}/\Mzs}|_{C_{4}} \cong 
	q_{p}^{*}\cO_{\PP^{2}}(-1)|_{C_{4}} 
	\cong \cO_{\PP^{2}}(-1)|_{q_{p}(C_{4})} \cong \cO_{\PP^{1}}(-2).
\]
Thus locally $\varphi_{A}$ is not isomorphic to a smooth blow-down and the image is singular along $L_{i}$. 
\end{proof}

\begin{remark}\label{rem:Igusaquartic}
In the literature, $\Vat$ has had several alternative descriptions. First of all, note that $\Vat \subset (\PP^{2})^{6}\git SL_{3}$ is the closure of the locus of configurations of six points on smooth conics. In \cite[p. 17, Example 3]{DO88} (also see \cite[Example 11.7]{Dol03}), it was proved that $(\PP^{2})^{6}\git SL_{3}$ is a double cover of $\PP^{4}$ which is ramified over a quartic hypersurface so called \textbf{Igusa quartic} (or \textbf{Castelnuovo-Richmond quartic}) $\cI_{4}$. It is defined by 
\[
	\sum_{i=1}^{6} X_{i} = 0, \quad 
	(\sum_{i=1}^{6}X_{i}^{2})^{2} - 4(\sum_{i=1}^{6}X_{i}^{4}) = 0
\]
in $\PP^{5}$. $\cI_{4}$ is exactly the locus of configurations on conics, thus $\Vat$ is isomorphic to $\cI_{4}$. On the other hand, $\cI_{4}$ is the Satake compactification $\overline{\cA}_{2}(2)$ of moduli space $\cA_{2}(2)$ of principal polarized abelian surfaces with level two structure \cite{Igu64}. To the author's knowledge, there have been no explicit constructions of a regular map from $\Mzs$ to $\overline{\cA}_{2}(2)$. 
\end{remark}

\begin{remark}\label{rem:normal}
From Proposition \ref{prop:geometryofquotient}, $\Vat$ is regular in codimension one. Since we can regard it as a complete intersection in $\PP^{5}$, it is Cohen-Macaulay, in particular, it has $S_{2}$-property. Thus by Serre's criterion, $\Vat$ is normal. 
\end{remark}

\section{Proof of the main theorem}\label{sec:proofmaththm}

In this section, we run Mori's program for $\Mzs$ with symmetric divisors.

\subsection{Stable base locus decomposition}
For an effective divisor $D$, the stable base locus $\mathbf{B}(D)$ is defined as
\[
	\mathbf{B}(D) = \bigcap_{m \ge 0}\mathrm{Bs}(mD),
\]
where $\mathrm{Bs}(D)$ is the set-theoretical base locus of $D$. As a first step toward Mori's program, we will compute the stable base locus decomposition of $\Mzs$, which dictates the different of birational models. 

\begin{proposition}\label{prop:stablebaselocus}
Let $D$ be a symmetric effective divisor on $\Mzs$. Then:
\begin{enumerate}
	\item If $D \in [-K_{\Mzs}, K_{\Mzs}+\frac{1}{3}\psi]$, 
	$D$ is semi-ample. 
	\item If $D \in (K_{\Mzs} + \frac{1}{3}\psi, B_{3}]$, 
	$\mathbf{B}(D) = B_{3}$
	\item If $D \in [B_{2}, -K_{\Mzs})$, $\mathbf{B}(D) = B_{2}$. 
\end{enumerate}
\end{proposition}

\begin{proof}
For $n = 6$, it is well-known that a divisor $D$ on $\Mzs$ is nef if and only if $D\cdot F \ge 0$ for all F-curves (\cite[Theorem 1.2]{KM96}). From Corollary \ref{cor:intersection}, it is straightforward to check that $\mathrm{Nef}(\Mzs)$ is generated by $-K_{\Mzs}$ and $K_{\Mzs} + \frac{1}{3}\psi$. Thus for item (1), it is sufficient to show that $-K_{\Mzs}$ and $K_{\Mzs} + \frac{1}{3}\psi$ are semi-ample. It is a direct consequence of the fact that $\Mzs$ is a Mori dream space (\cite{Cas09}), but in this case furthermore we can write these divisors as pull-backs of ample divisors. $K_{\Mzs} + \frac{1}{3}\psi$ is a scalar multiple of the pull-back of an ample line bundle on $(\PP^{1})^{6}\git SL_{2}$ by Example \ref{ex:Kapranovmorphism}. From Lemma \ref{lem:canpolarization}, $-K_{\Mzs}$ is a scalar multiple of a semi-ample divisor $D_{A}$. So both of them are semi-ample.

If $D \in (K_{\Mzs}+\frac{1}{3}\psi, B_{3}]$, since $K_{\Mzs}+\frac{1}{3}\psi$ is semi-ample, $\mathbf{B}(D) \subset B_{3}$. On the other hand, by Corollary \ref{cor:intersection}, $F_{1,1,1,3}\cdot D < 0$ thus $F_{1,1,1,3} \subset \mathbf{B}(D)$. But $F_{1,1,1,3}$ covers an open dense subset of $B_{3}$. So $\mathbf{B}(D) = B_{3}$. 

Finally, if $D \in [B_{2}, -K_{\Mzs})$, $\mathbf{B}(D) \subset B_{2}$ because $-K_{\Mzs}$ is semi-ample. $C_{4} \cdot D < 0$ implies $C_{4} \subset \mathbf{B}(D)$. Since $C_{4}$ covers an open dense subset of $B_{2}$, $\mathbf{B}(D) = B_{2}$.
\end{proof}

We can summarize the above result by Figure \ref{fig:stablebaselocus}. 

\begin{figure}[!ht]
\begin{tikzpicture}[scale=0.6]
	\draw[->][line width=1.5pt][lightgray] (5, 0) -- (10, 0);
	\draw[<-][line width=1.5pt] (0,0) -- (5,0);
	\draw[->][line width=1.5pt][lightgray] (5,0) -- (5,5);
	\draw[->][line width=1.5pt] (5,0) -- (9.8,2);
	\draw[->][line width=1.5pt] (5,0) -- (0.2, -1);
	\draw[->][line width=1.5pt] (5,0) -- (9.6, 3);
	\node at (11,0) {$K_{\Mzs}$};
	\node at (-1.3, 0) {$-K_{\Mzs}$};
	\node at (5,5.5) {$\psi$};
	\node at (10.3, 2) {$B_{3}$};
	\node at (-0.5,-1) {$B_{2}$};
	\node at (11.5, 3.2) {$K_{\Mzs} + \frac{1}{3}\psi$};
	\node at (9, 2.1) {$B_{3}$};
	\node at (1, -0.4) {$B_{2}$};
	\node at (5,3) {$\emptyset$};
\end{tikzpicture}
\caption{Stable base locus decomposition of $\Mzs$}\label{fig:stablebaselocus}
\end{figure}

\subsection{Mori's program for $\Mzs$}

Now we can perform Mori's program of $\Mzs$ for all symmetric divisors. 

\begin{theorem}\label{thm:MoriprogramMzs}
Let $D$ be a symmetric effective divisor on $\Mzs$. Then:
\begin{enumerate}
	\item If $D \in (-K_{\Mzs}, K_{\Mzs}+\frac{1}{3}\psi)$, 
	then $\Mzs(D) \cong \Mzs$. 
	\item If $D \in [K_{\Mzs}+\frac{1}{3}\psi, B_{3})$, 
	then $\Mzs(D) \cong (\PP^{1})^{6}\git SL_{2}$. 
	\item If $D \in (B_{2}, -K_{\Mzs}]$, then $\Mzs(D)$ is 
	$\Vat$ where $A = (\frac{1}{2}, \cdots, 
	\frac{1}{2})$. 
	\item Both $\Mzs(B_{2})$ and $\Mzs(B_{3})$ are a point. 
\end{enumerate}
\end{theorem}

\begin{proof}
If $D \in (-K_{\Mzs}, K_{\Mzs} + \frac{1}{3}\psi)$ it is ample by Proposition \ref{prop:stablebaselocus}. Thus $\Mzs(D) \cong \Mzs$. 

If $D \in [K_{\Mzs} + \frac{1}{3}\psi, B_{3})$, then $D = K_{\Mzs} + \frac{1}{3}\psi + aB_{3}$ for some $a > 0$. By taking a sufficiently large multiple, we may assume that $D$ is a linear combination of integral divisors. Note that $K_{\Mzs}+\frac{1}{3}\psi$ is $\rho^{*}L$ for $\rho : \Mzs \to (\PP^{1})^{6}\git SL_{2}$ of an ample line bundle $L$ (Example \ref{ex:Kapranovmorphism}), and $B_{3}$ is the exceptional locus of $\rho$. Thus by \cite[Lemma 7.11]{Deb01}, 
\begin{eqnarray*}
	H^{0}(\Mzs, \cO(mD)) & = & 
	H^{0}(\Mzs, \cO(m(K_{\Mzs}+\frac{1}{3}\psi+aB_{3}))) \\
	&\cong & H^{0}((\PP^{1})^{6}\git SL_{2}, L^{m}).
\end{eqnarray*}
Therefore
\[
	\Mzs(D) \cong \proj\bigoplus_{m \ge 0}H^{0}(\Mzs, \cO(mD)) \cong
	\proj \bigoplus_{m \ge 0}
	H^{0}((\PP^{1})^{6}\git SL_{2}, L^{m})
	\cong (\PP^{1})^{6}\git SL_{2}.
\]

Similarly, if $D \in (B_{2}, -K_{\Mzs}]$, then $D = -K_{\Mzs} + bB_{2}$ for some $b > 0$ and $B_{2}$ is the exceptional locus of $\varphi_{A} : \Mzs \to \Vat$. Also $-K_{\Mzs}$ is a pull-back of an ample line bundle $M$ on $\Vat$. By the same argument, 
\[
	\Mzs(D) \cong \proj \bigoplus_{m \ge 0}
	H^{0}(\Vat, M^{m}),
\]
which is the normalization of $\Vat$. But $\Vat$ is normal by 
Remark \ref{rem:normal} so it is isomorphic to its normalization.

The last assertion is obvious since both $B_{2}, B_{3}$ are fixed divisors. 
\end{proof}

\begin{remark}
Therefore in symmetric Mori's program on $\Mzs$, there is no flip. In Mori's full program on $\Mzs$ there must be many flips because we can construct many modular small contractions of $\Mzs$ by using a general construction described in \cite{GJM13}. 

On the other hand, for $n \ge 7$, even for symmetric Mori's program many flips should appear. 
\end{remark}

\begin{remark}
For divisors of the form $K_{\Mzn} + \alpha B$ with $\alpha \le 1$, Mori's program can be regarded as an analogy with the Hassett-Keel program for $\cMg$. In this direction, Mori's program is done by many authors \cite{Sim08, AS08, FS11, KM11} for all $n$. When $n = 6$ it covers half of the symmetric effective cone. 

For the subcone of effective cone generated by $K_{\Mzn}$ and $\psi_{i}$-classes has been studied intensively. There is a general picture for non-symmetric weight data and even for higher genera cases. Consult \cite{Moo11b}. 
\end{remark}

\begin{remark}\label{rem:Segrecubic}
It is well-known that $(\PP^{1})^{6}\git SL_{2}$ is isomorphic to \textbf{Segre cubic} $\cS_{3}$, which is defined by
\[
	\sum_{i=1}^{6}X_{i} = 0, \quad \sum_{i=1}^{6}X_{i}^{3} = 0
\]
in $\PP^{5}$ (\cite[Example 11.6]{Dol03}). One fascinating fact is that Segre cubic is projectively dual to Igusa quartic $\cI_{4}$ (\cite[Remark I.3]{DO88}), as a hypersurface in $\PP^{4}$. It would be interesting if this projective dual map can be described concretely in terms of $\Mzs$. 
\end{remark}

\section{Mori's program for $\cMt$}\label{sec:MoriprogramcMt}

Let $\cMt$ be the moduli stack of genus two stable curves. Essentially all birational contractions of $\cMt$ were described in \cite{Has05}, even though Mori's program was described for only one half of the effective cone of $\cMt$. Indeed, since all genus two smooth curves are hyperelliptic, the coarse moduli space $\Mt$ is isomorphic to $\Mzs/S_{6}$ (\cite[Corollary 2.5]{AL02}). Thus Theorem \ref{thm:MoriprogramMzs} gives Mori's full program for $\Mt$. Also since $\mathrm{Pic}(\Mt)_{\QQ} \cong \mathrm{Pic}(\cMt)_{\QQ}$, this result can be regarded as Mori's program of $\cMt$. 

First of all, we have natural isomorphisms
\[
	\mathrm{Pic}(\cMt)_{\QQ} \cong
	\mathrm{Pic}(\Mt)_{\QQ}
	\cong \mathrm{Pic}(\Mzs)_{\QQ}^{S_{6}}.
\]
For the coarse moduli map $q : \cMt \to \Mt$, we have the first isomorphism $q^{*}: \mathrm{Pic}(\Mt)_{\QQ} \to \mathrm{Pic}(\cMt)_{\QQ}$. This is an isomorphism only if we take $\QQ$-Picard groups. The second isomorphism comes from $\pi^{*}: \mathrm{Pic}(\Mt) \to \mathrm{Pic}(\Mzs)^{S_{6}}$ where $\pi : \Mzs \to \Mt$ is the $S_{6}$-quotient map. By following the notations in \cite{Has05}, we denote by $\delta_{0}$ (resp. $\delta_{1}$) the boundary divisor of irreducible nodal curves (resp. that of union of two elliptic curves respectively) on $\cMt$. Let $\Delta_{0}, \Delta_{1}$ be corresponding boundary divisors on the coarse moduli space $\Mt$. Let $\lambda$ be Hodge class on $\Mt$ also its pull-back on $\cMt$. The effective cone of $\cMt$ (resp. $\Mt$) is generated by $\delta_{0}$ and $\delta_{1}$ (resp. $\Delta_{0}$ and $\Delta_{1}$). Since $q : \cMt \to \Mt$ is ramified along $\Delta_{1}$, $q^{*}(\Delta_{0}) = \delta_{0}$, $q^{*}(\Delta_{1}) = 2\delta_{1}$. 

The following simple lemma says about the relation between $S_{6}$-symmetric Mori's program of $\Mzs$ and Mori's program of $\Mt$. For a projective variety $X$ and a divisor $D$, let
\[
	R(X, D) := \bigoplus_{m \ge 0}H^{0}(X, \cO(mD))
\]
be the section ring. 

\begin{lemma}
Let $G$ be a finite group acting on a projective variety $X$. And let $\pi : X \to X/G$ be the quotient map and $D$ be a $\QQ$-Cartier divisor on $X/G$. Assume that $R(X, \pi^{*}D)$ is finitely generated and $Y := \proj R(X, \pi^{*}D)$. Then $R(X/G, D)$ is finitely generated and $\proj R(X/G, D) = Y/G$. 
\end{lemma}

\begin{proof}
We may assume that $D$ is a Cartier divisor, because $\proj R(X, D)$ does not change after replacing $D$ by $kD$. Note that $H^{0}(X/G, \cO(mD)) \cong H^{0}(X, \pi^{*}\cO(mD))^{G}$. Thus 
\[
	R(X/G, D) = \bigoplus_{m \ge 0}H^{0}(X/G, \cO(mD))
	\cong \bigoplus_{m \ge 0}H^{0}(X, \pi^{*}\cO(mD))^{G}
	= R(X, \pi^{*}D)^{G}.
\]
Since $G$ is a finite group and $R(X, \pi^{*}D)$ is finitely generated, the $G$-invariant subring is finitely generated, too. By the definition of projective quotient, $\proj R(X/G, D) \cong \proj R(X, \pi^{*}D)^{G} = Y/G$. 
\end{proof}

\begin{theorem}\label{thm:MoriprogramMtwo}
Let $D$ be an effective divisor on $\Mt$. Then:
\begin{enumerate}
	\item If $D \in (\lambda, \Delta_{0}+6\Delta_{1})$, 
	$\Mt(D) \cong \Mt$. 
	\item If $D \in [\Delta_{0}+6\Delta_{1}, \Delta_{1})$, 
	$\Mt(D) \cong \PP^{6}\git SL_{2}$. 
	\item If $D \in (\Delta_{0}, \lambda]$, $\Mt(D) \cong 
	\overline{\cA}_{2}^{\mathrm{Sat}}$, Satake compactification of 
	$\cA_{2}$. 
	\item Both $\Mt(\Delta_{0})$ and $\Mt(\Delta_{1})$ are a point.
\end{enumerate}
\end{theorem}

\begin{proof}
First of all, for the quotient map $\pi : \Mzs \to \Mt$, set theoretically $\pi^{-1}(\Delta_{0}) = B_{2}$ and $\pi^{-1}(\Delta_{1}) = B_{3}$. Note that $S_{6}$ acts freely on a general point of $\Mzs - (B_{2} \cup B_{3})$ and a general point of $B_{3}$, but there is an order two stabilizer on a general point $B_{2}$, which changes marked points on the irreducible component with two marked points. Thus $\pi^{*}(\Delta_{0}) = 2B_{2}$ and $\pi^{*}(\Delta_{1}) = B_{3}$. 

Since $\lambda \equiv \frac{1}{10}(\Delta_{0}+\Delta_{1})$ (\cite[Exercise 3.143]{HM98}), $\pi^{*}\lambda = \frac{1}{5}B_{2}+\frac{1}{10}B_{3} = -\frac{1}{2}K_{\Mzs}$. Also $\pi^{*}(\Delta_{0}+6\Delta_{1}) = 2B_{2}+6B_{3}= 15(K_{\Mzs}+\frac{1}{3}\psi)$. Thus the decomposition of the effective cone of $\Mt$ on the statement is exactly the image of the stable base locus decomposition of $\Mzs$. So we obtain the result from Theorem \ref{thm:MoriprogramMzs}. For example, for $D \in (\lambda, \Delta_{0}+6\Delta_{1})$, $\Mt(D) = \Mzs(\pi^{*}D)/S_{6} = \Mzs/S_{6}=\Mt$. For $D \in [\Delta_{0}+6\Delta_{1}, \Delta_{1})$, $\Mt(D) = \Mzs(\pi^{*}D)/S_{6} = (\PP^{1})^{6}\git SL_{2}/S_{6} = (\PP^{1})^{6}/S_{6}\git SL_{2} = \PP^{6}\git SL_{2}$ since $SL_{2}$-action and $S_{6}$-action commute. Finally, we already know that $\Mt(\lambda) = \overline{\cA}_{2}^{\mathrm{Sat}}$ from the definition of $\lambda$ class. Now for $D \in (\Delta_{0}, \lambda]$, $\Mt(D) \cong \Mzs(\pi^{*}D)/S_{6}$ is independent from the choice of $D$. Therefore $\Mt(D) \cong \overline{\cA}_{2}^{\mathrm{Sat}}$. 
\end{proof}

Finally, we can obtain an alternative proof of the main theorem of \cite{Has05}, as a restatement of Theorem \ref{thm:MoriprogramMtwo}. Note that 
$q^{*}(\Delta_{0}+6\Delta_{1}) = \delta_{0}+12\delta_{1}$. 

\begin{theorem}\label{thm:MoriprogramcMt}
Let $D$ be an effective divisor on $\cMt$. Then:
\begin{enumerate}
	\item If $D \in (\lambda, \delta_{0}+12\delta_{1})$, 
	$\cMt(D) \cong \Mt$.
	\item If $D \in [\delta_{0}+12\delta_{1}, \delta_{1})$, 
	$\cMt(D) \cong \PP^{6}\git SL_{2}$.
	\item If $D \in (\delta_{0}, \lambda]$, $\cMt(D) \cong 
	\overline{\cA}_{2}^{\mathrm{Sat}}$, Satake compactification of 
	the moduli space of principally polarized abelian surfaces.
	\item Both $\cMt(\delta_{0})$ and $\cMt(\delta_{1})$ are a point.
\end{enumerate}
\end{theorem}

\begin{remark}
\begin{enumerate}
	\item The divisor classes of the form $K_{\cMt}+\alpha \delta$ 
	are more familiar for people interested in Hassett-Keel program. 
	In this setup, 
	Theorem \ref{thm:MoriprogramcMt} is translated as the following: For 
	$2 > \alpha > 9/11$, $\cMt(\alpha) := \cMt(K_{\cMt}+\alpha \delta)
	\cong \Mt$. For $9/11 \ge \alpha > 7/10$, $\cMt(\alpha) \cong 
	\PP^{6}\git SL_{2}$. For $\alpha \ge 2$, $\cMt(\alpha) \cong 
	\overline{\cA}_{2}^{\mathrm{Sat}}$. Finally, $\cMt(7/10)$ is a point. 
	But it does not cover the part $[\delta_{0}, \delta_{0}+\delta_{1}]$ of 
	the effective cone . 
	\item It is well-known that $\PP^{6}\git SL_{2} \cong \PP(2,4,6,10)$ 
	(\cite[Section 10.2]{Dol03}).
\end{enumerate}
\end{remark}

%%%%%%%%%%%%%%%%%%%%%%%%%%%%%%%%%%%%

\bibliographystyle{alpha}
\bibliography{Library}

\end{document}